\documentclass{article}
\usepackage[utf8]{inputenc}
\usepackage{amsfonts,amssymb} 
\usepackage{amsthm}
\usepackage{cite}
\usepackage{mathtools}
\usepackage{graphicx}
\usepackage{mathtools}
\usepackage{authblk}
\usepackage{hyperref}
\hypersetup{
hidelinks,
colorlinks=true,
linkcolor=red,
citecolor=blue,
urlcolor=black
}

\begin{document}
\newtheorem{theorem}{Theorem}[section]
\newtheorem{definition}{Definition}
\newtheorem{remark}{Remark}
\newtheorem{proposition}[theorem]{Proposition}
\newtheorem{lemma}[theorem]{Lemma}
\newtheorem{assumption}{Assumption}
\newtheorem{corollary}[theorem]{Corollary}

\title{Entropy Continuity of Lyapunov Exponents for Non-flat 1-dimensional Maps}

\author{Hengyi Li}

\affil{Institut de Mathématique d'Orsay, CNRS - UMR 8628
Universit\'e Paris-Saclay, 91405 Orsay, France. \emph{E-mail address:}
\tt{hengyi.li@universite-paris-saclay.fr}}

\maketitle

\abstract{We show that the continuity property of Lyapunov exponents proved in \cite{BCS-Exponents} for smooth surface diffeomorphisms extends to smooth interval maps, in the case when the map only has non-flat critical points and the entropies converging to the topological entropy.

The result we obtained is stronger than the continuity of Lyapunov exponents, in particular, we prove the uniform integrability of Lyapunov exponents over entropies.}

\section{Introduction}

Lyapunov exponents are key invariants of smooth ergodic theory. They are upper semicontinuous, but usually not lower continuous functions of the invariant probability measure \cite{V2}.
The goal of this work is to understand when the convergence of the entropy implies convergence of the Lyapunov exponents for smooth interval maps as was established for surface diffeomorphisms in \cite{BCS-Exponents}. 

In this paper, we consider the simplest case when the map has only non-flat critical points and the entropy converges to the topological entropy. We obtain convergence of the Lyapunov exponents in theorem \ref{main}. More precisely, we establish uniform integrability of $\log |f'|$ in theorem \ref{gap}. For a general survey on this phenomena, see \cite{V1}, \cite{S}.

The one-dimensional case with non-flat critical points turns out to exhibit quite different difficulties than the surface diffeomorphism case, but the difficulties are quite different. This is also the simplest case of non-invertible and non-uniformly hyperbolic systems \cite{Q}. They come mostly from the lack of bounded distorsion with respect to the differential. Namely, unlike the diffeomorphism case, we no longer have some uniform $\varepsilon_0$ such that for all $0<\varepsilon<\varepsilon_0$, for all $x$, $f(B(x,\varepsilon))\subseteq B(fx,2|f' x|\varepsilon)$, where $B(x,\varepsilon)$ is the ball centered at $x$. 

One possible application of theorem \ref{main} is to prove not only the finiteness \cite{tail}, but also the exponential mixing of the measures of maximal entropy for $C^{\infty}$ one-dimensional systems with positive topological entropy and finite critical set \cite{M1}.

\subsection{Statement of our main result}

We let $X=[0,1]$ or $\mathbb{R}/\mathbb{Z}$, and  $f:X\to X$ be a $C^{\infty}$ smooth map. Given $x\in X$, we say $f$ is non-flat at $x$, or simply that $x$ is non-flat, if there exists $d>0$ such that $f^{(d)}(x)\neq 0$. We assume that for all $x\in X$ such that $f'(x)=0$, $x$ is non-flat.

In this text, a measure will always mean an $f$-invariant Borel probability measure on $X$. For a measure $\mu$, we denote by $h(f,\mu)$ or just $h(\mu)$, the entropy of $\mu$. We denote by $h_{\rm top}(f)$ the topological entropy of $f$.

We shall consider sequences of measures of the following type.

\begin{definition}
Let $(\mu_k)_{k\in \mathbb{N}^{\ast}}$ be a sequence of measures on $X$. We say that it \emph{approximates $f$ in entropy,} if 
\begin{itemize}
\item[1.] each $\mu_k$ is ergodic for $f$;
\item[2.] the sequence converges in the weak topology: $\mu_k \xrightarrow[]{\ast} \mu$ for some measure $\mu$;
\item[3.] the entropies converge to the topological entropy: $\lim_{k\rightarrow \infty }h(f,\mu_k)=h_{\rm top}(f).$
\end{itemize}
\end{definition}

For any probability measure $\nu$ on $X$, define its Lyapunov exponent to be
$$
\lambda(\nu):=\int \log |f'| d\nu.
$$

Our main result is the following.

\begin{theorem} \label{main}
Let $f:X\to X$ be a smooth map with only non-flat critical points and let $(\mu_k)_{k\in \mathbb{N}}$ be a sequence of probability measures on $X$ approximating $f$ in entropy. Denote its weak limit by $\mu$. If $h_{\rm top}(f)>0$, then
\begin{equation}
  \tag{$\ast$}
  \lim_{k\rightarrow \infty} \lambda(\mu_k)=\lambda(\mu).
  \label{eqn:ast}
\end{equation}
\end{theorem}

\subsection{Strategy: entropy and uniform integrability defect}

Our overall strategy, like in \cite{BCS-Exponents}, is to deduce from any failure of continuity of the exponent a non-trivial bound on the entropy. As one could expect, the critical points play a role similar to the homoclinic tangencies there. In contrast to the decomposition of theorem D of \cite{BCS-Exponents}, we are only able to quantify numerically the lack of uniform integrability of $\log|f'|$ with respect to the $\mu_k$, $k\ge1$ as follows.

\begin{definition}
The \emph{uniform integrability defect} of $(\mu_k)_{k\in\mathbb{N}^+}$ is
 $$
     \alpha:=\lim_{\delta\to0} \limsup_{k\to\infty} \int_{\{|f'|<\delta\}} - \log|f'|\, d\mu_k.
 $$ 
\end{definition}

It is well known that $\lambda(\nu)$ is upper semi-continuous with respect to $\nu$. As $-\log|f'|$ is continuous on $\{ x: |f'(x)|\geq \delta \}$, $\int_{\{ |f'|\geq \delta \}}-\log|f'|d\mu_k\rightarrow \int_{\{ |f'|\geq \delta \}} -\log|f'|d\mu $. One then sees that:

\begin{lemma}\label{defcon}
The uniform-defect $\alpha$ of a sequence of measures gives the defect in lower semicontinuity of the Lyapunov exponent:
 $$
       \alpha=\lambda(\mu)-\liminf_{k\to\infty} \lambda(\mu_k).
 $$
\end{lemma}

Setting $\Lambda(f):=\sup|f'|$, and $\lambda(f):=\log \Lambda$, we get the following entropy bound:

\begin{theorem}\label{Main}
Let $f:X\rightarrow X$ be a smooth map with non-flat critical points. Let $(\mu_k)_{k\in\mathbb{N}^+}$ be a sequence of ergodic measures converging in the weak topology to some measure $\mu$. If $\alpha$ is the uniform integrability defect then:
 $$
      \limsup_{k\to\infty} h(\mu_k) \le \left(1-\frac{\alpha}{\lambda(f)}\right) h_{\rm top}(f).
 $$
\end{theorem}

Note that the two previous statements imply our main result, Theorem~\ref{main}.

\subsection{Organization of the paper}

In section 2, we define and investigate shadowing intervals. These intervals are orbit segments approximated by critical points. We prove that the portion of these intervals among $[0,n[$ converges to the uniform defect $\frac{\alpha}{\lambda(f)}$ (theorem \ref{main2}).

In section 3, we deduce an entropy bound (theoren \ref{gap}) from the portion of shadowing intervals among $[0,n[$.

\subsection{Notations}

Let $C(f):=\{ c\in X: f'(c)=0 \}$. Here $C(f)$ is a finite set, since each $c$ is non-flat. For each $c\in C(f)$, write $d_c:=\min \{ d\in \mathbb{N}^{+}: f^{(d)}(c)\neq 0 \}$, and
$$
p_c(x):=\frac{f^{(d_c)}(c)}{(d_c-1)!}(x-c)^{d_c-1},
$$
the $d_c-1$-th Taylor polynomial of $f'$ at $c$.
Then, define $g_c(x)$ as the unique smooth function such that
\begin{equation}
  \tag{$1.1$}
  f'(x)=p_c(x)\cdot g_c(x).
  \label{eqn:2ast}
\end{equation}
Define $K(f):=\max_{c\in C(f)}\{ d_c \}$. When $f$ is clear from context, we abbreviate $K:=K(f)$. Note that $g_c(c)=1$.

We abbreviate $f'x:=f'(x)$, $x_a:=f^a x:=f^a(x)$ for all $x\in X$, for all $a\in \mathbb{N}$.

Since $C(f)$ is a finite set, for $\varepsilon<\frac{1}{2}\min\{ d(c_1,c_2): c_1,c_2\in C(f) \}$ and for all $x\in X$, the number of points in $(x-\varepsilon,x+\varepsilon)\cap C(f)\leq 1$. Given $a\in \mathbb{N}$, when $(x_a-\varepsilon,x_a+\varepsilon)\cap C(f)\neq \emptyset$, we let $c_a$ be the unique point in the intersection.

Fix $\varepsilon'_1>0$ such that for all $c\in C(f)$, $|f'|_{(c-2\varepsilon'_1,c)}|$ is monotone decreasing and $|f'|_{(c,c+2\varepsilon'_1)}|$ is monotone increasing. Note that $(c-2\varepsilon'_1,c+\varepsilon'_1)\cap C(f)=\{ c \}$ for all $c\in C(f)$. Fix $\varepsilon''_1>0$ such that for all $c\in C(f)$, for all $x,y\in (c-2\varepsilon''_1,c+2\varepsilon''_1)$, $\frac{1}{2}<\frac{|g_c(x)|}{|g_c(y)|}<2$. Finally, let
\begin{equation}
  \tag{$1.2$}
  \varepsilon_1:=\min\{ \varepsilon'_1,\varepsilon''_1 \},
  \label{eqn:3ast}
\end{equation}

Denote by $\mathbb{P}(f)$ the space of $f$-invariant, Borel, probability measures on $X$, and $\mathbb{P}_{\rm erg}(f)$ the subspace of ergodic measures in $\mathbb{P}(f)$.

Throughout this paper, we consider integer intervals, such as $[a,b[:=\{ k\in \mathbb{N}: a\leq k<b, a\in \mathbb{N}, b\in \mathbb{R}_{+} \}$ for $-\infty< a <b\leq \infty$. For every finite subset $K\subseteq \mathbb{N}$, we denote by $K^{-}:=\min K$ and $K^{+}:=\max K$.

\subsection*{Acknowledgements}

The author is supported by the Chinese Scholarship Consul and the Hadamard Doctoral School of Mathematics.

\section{Shadowing Intervals}

\begin{definition}
Given $\varepsilon>0$, an integer interval $I=[a,b[$ is an {\emph{$\varepsilon$-shadowing interval for $x\in X$}}, if there exists $c\in C(f)$, such that for all $0\leq h<b-a$,
$$d(x_{a+h},f^h c)<\varepsilon.$$
For all $L\in \mathbb{N}^{+}$, an $\varepsilon$-shadowing interval $[a,b[$ for $x$ is {\emph{$(L,\varepsilon)$-shadowing}}, if $b-a\geq L$.
\end{definition}

In this section, we will prove the following theorem.

\begin{theorem}\label{main2}
For all $\varepsilon>0$, $L\in \mathbb{N}^{+}$, there exists $\delta>0$, such that for all $\nu\in \mathbb{P}_{\rm erg}(f)$, with $h(\nu)>0$, for $\nu$-a.e. $x\in X$, there exists a collection $\mathcal{S}(x)$ of finite integer intervals, satisfying
\begin{itemize}
\item[0.]  the map $x\mapsto \mathcal{S}(x)$ is measurable on $X$: for each $n\in \mathbb{N}$, for each collection $J$ of subsets of $[0,n[$, $\mathcal{S}^{-1}_n(J):=\{ x\in X: \{I\cap [0,n[: I \in \mathcal{S}(x)\}=J \}$ is measurable;
\item[1.] the intervals in $\mathcal{S}(x)$ are $(L,\varepsilon)$-shadowing and pairwise disjoint;
\item[2.] denote by $\bigcup \mathcal{S}(x):=\bigcup_{I\in \mathcal{S}(x)} I$. For all $n\in \mathbb{N}^{+}$, 
$$
\left|\bigcup \mathcal{S}(x)\cap [0,n[ \right|\geq (1-\frac{1}{L})\sum_{i=0}^{n-1} -\frac{\log_{\delta}|f'x_i|}{\lambda}-g_{\varepsilon,L,n}(x),
$$
where
$$
\log_{\delta} t:=
\left\{
\begin{aligned}
&0, &{\text{if $t\geq \delta$}};\\
& \log t, &{\text{otherwise}},
\end{aligned}
\right.
$$
and $g_{\varepsilon,L,n}: X\rightarrow \mathbb{R}$, depending on $\varepsilon$, $L$ and $n$, is measurable and satisfies for $\nu$-a.e. $x$,
$$
\lim_{n\rightarrow \infty}\frac{1}{n}g_n(x)=0.
$$
\end{itemize}
\end{theorem}

\subsection{Basic Properties}

\begin{lemma}\label{delta}
For all $L\in \mathbb{N}^+$, $0<\varepsilon<\varepsilon_1$, there exists $\delta_1>0$ such that for all $0<\delta<\delta_1$, for all $x\in X$, and for all $a\in \mathbb{N}$, if $|f'x_a|<\delta$, the integer interval $[a,a+l(a)[$ is a $(L,\varepsilon)$-shadowing interval, where
$$l(a):=\frac{-\log |f'x_a|}{\lambda(f)}.$$
\end{lemma}
\begin{proof}
Given $L>0$, $0<\varepsilon<\varepsilon_1$, choose $\delta'_1\in (0,1)$ such that $|f'x|<\delta'_1$ implies there exists a unique $c\in C(f)$, with $d(x,c)<\varepsilon$. Then, define $\delta_1:=\min\{ \delta'_, \Lambda(f)^{-L}\}$.

Assume $a\in \mathbb{N}$ satisfies $|f'x_a|<\delta_1$. We check that $[a,a+l(a)[$ is $(L,\varepsilon)$-shadowing.

Firstly, note that $l(a)\geq L$. Indeed,
$$
l(a)=\frac{-\log |f'x_a|}{\lambda(f)}>\frac{-\log \delta_1}{\lambda(f)}> \frac{ L\log \Lambda(f)}{\lambda(f)}=L.
$$

Secondly, observe that $[a,a+l(a)[$ is $\varepsilon$-shadowing:

Since $\delta_1<\delta'_1$, there exists a unique $c_a\in C(f)$ such that $d(x_a,c_a)<\varepsilon$. Now for all $0<h<l(a)$,
$$
\begin{aligned}
d(x_{a+h},f^h c_a)&\leq \Lambda(f)^{h-1}d(x_{a+1},f c_a)\\
&\leq \Lambda(f)^{h-1} \cdot |f'x_a|d(x_a,c_a)\leq \Lambda(f)^{l(a)-1}\cdot |f'x_a| d(x_a,c_a)\\
&\leq \Lambda(f)^{-\log|f'x_a| / \lambda(f)}|f'x_a| d(x_a,c_a)\leq d(x_a,c_a) <\varepsilon,
\end{aligned}
$$
the second inequality follows from the fact that $d(x_a,c_a)<\varepsilon<\varepsilon_1$ and $|f'|$ is monotone decreasing from $x_a$ to $c_a$; the last inequality follows from our choice of $\delta<\delta'_1$, which guarantees $\Lambda(f)^{-\log|f'x_a| / \lambda(f)}|f'x_a|\leq 1$. 
\end{proof}

In the following sections, since the parameters $\varepsilon>0$, $\delta>0$ and $x\in X$ are fixed, we omit them from the notations. 

We define $A(x):=\{ a\in \mathbb{N}: |f'x_a|<\delta \}$, and
$$
F( n)=
\left\{
\begin{aligned}
&\Lambda(f), &{\text{if $|f'x|\geq \delta$}};\\
&2^{K+1}\cdot |f' x|, &{\text{otherwise}}.
\end{aligned}
\right.
$$
When the base point $x$ is clear from context, we abbreviate $A:=A(x)$.

\begin{definition}
An {\emph{$(\varepsilon, \delta)$-piecewise shadowing interval for $x\in X$}}, abbreviated as a $(\varepsilon,\delta)$-PSI or just a PSI, is a sequence of consecutive $\varepsilon$-shadowing intervals
$$
I:=\{[a_1,a_2[, [a_2,a_3[, \cdots ,[a_n,a_{n+1}[\},
$$
such that
\begin{itemize}
\item[0.] $a_1<a_2<\cdots<a_{n+1}$;
\item[1.] for all $1\leq i\leq n$, $a_i\in A$;
\item[2.] $\forall a_1\leq k<a_{n+1}$,
$$
G^I(k):=\prod_{a_1\leq j\leq k}F( j)<1.
$$
\end{itemize}
The {\emph{support of $I$}}, $ \operatorname{supp} I$, is the union $[a_1,a_{n+1}[$. We also set
$$
G^I(a_{n+1}):=\prod_{a_1\leq j\leq a_{n+1}}F( j)
$$
\end{definition}

Note that a PSI is not an interval but a collection of intervals. 

Note that if $\{[a_1,a_2[, \cdots, [a_n,a_{n+1}[\}$ is $\varepsilon$-piecewise shadowing, so will be $\{[a_1,a_2[, \cdots ,[a_s,b[\}$ for all $s\leq n$ and $b\in [a_s,a_{s+1}]$.

\begin{definition}
A piecewise shadowing interval $I=\{[a_1,a_2[,\cdots, [a_n,a_{n+1}[\}$ is right maximal, if $G^I(a_{n+1})\geq 1$.
\end{definition}

\begin{lemma}\label{last}
Let $I=\{[a_1,a_2[, \cdots, [a_n,a_{n+1}[\}$ be a right maximal PSI. Then $a_{n+1}\notin A$.
\end{lemma}
\begin{proof}
Since $I$ is a PSI, $ G^I(a_{n+1}-1)<1$. If $a_{n+1}\in A$, then $F(a_{n+1})=2^{K+1}\cdot |f'x_{a_{n+1}}|<1$, and
$$
G^I(a_{n+1})=G^I(a_{n+1}-1)\cdot F(a_{n+1})<G^I(a_{n+1})=G^I(a_{n+1}-1)<1,
$$
contradicting our assumption that $I$ is right maximal.
\end{proof}

\begin{lemma}\label{union}
Let $I_1=\{[a_1,a_2[, \cdots, [a_n,a_{n+1}[\}$, $I_2=\{[b_1,b_2[, \cdots, [b_m,b_{m+1}[\}$ be two PSI's, such that $[a_1,a_{n+1}[ \cap [b_1,b_{m+1}[\neq \emptyset$. Denote by
$$
\{ s_1<s_2<\cdots <s_{l+1} \}=\{ a_1,\cdots a_n\}\cup\{ b_1,\cdots b_m \}\cup \max\{ a_{n+1},b_{m+1} \},
$$
then $J:=\{[s_1,s_2[, \cdots, [s_l,s_{l+1}[\}$ is again a PSI. We denote it by $I_1 \cap I_2$.
\end{lemma}
%\begin{proof}

%If the intervals are nested, say $I_1\subseteq I_2$, then one verifies easily that $J$ is a PSI.

%Assume without loss of generality that $b_{m+1}=s_{l+1}> a_{n+1}$ and $b_1> a_1=s_1$. Note that $b_1\in [a_s,a_{s+1}[$ for some $s=1,\cdots, n$. 

%Firstly, by construction, for all $[s_i,s_{i+1}[$, there exists $[a_j,a_{j+1}[$ or $[b_k,b_{k+1}[$ such that $s_i=a_j, s_{i+1}\leq a_{j+1}$, or $s_i=b_k, s_{i+1}\leq b_{k+1}$. So $[s_i,s_{i+1}[$ is a shadowing interval.

%Notice that $s_1,\cdots , s_l\in A$ by construction. For all $s_j\leq k<s_{j+1}$, and $a_1\leq k\leq b_1$, $G^J(k)=G^{I_1}(k)<1$. For all $s_j\leq k<s_{j+1}$, and $k>b_1$,
%$$
%\begin{aligned}
%G^J(k)&=\prod_{s_1\leq j<k}F(j)=\prod_{s_1\leq j<b_1}F(j)\cdot \prod_{b_1\leq j<k}F(j)\\
%&=G^{I_1}(b_1)\cdot G^{I_2}(k)<1\cdot 1=1,
%\end{aligned}
%$$
%where the last inequality follows from $b_1\in [a_s,a_{s+1}[$ and $I_1, I_2$ are both piecewise shadowing. 
%\end{proof}
\begin{proof}
Up to changing notation, we assume $a_1\leq b_1$.

Note that for $j\leq l$, $s_j=a_i$ with $i\leq n$, or $b_j$ with $j\leq m$. Thus 
\begin{itemize}
\item[1.] all $s_j\in A$;
\item[2.] $[s_j,s_{j+1}[$ is a segment of some $[a_i,a_{i+1}[$ or $[b_i,b_{i+1}[$, with $s_i=a_i$ or $b_i$, hence a shadowing interval.
\end{itemize}

It remains to check that $G^J(k)<1$ for all $k\in [s_1,s_{l+1}[$. Now if $k\leq a_{n+1}$, then $G^J(k)=G^{I_1}(k)<1$. If $k>a_{n+1}$, then $b_1<k<b_{m+1}$, and 
$$
G^J(k)=G^{I_1}(b_1)\cdot G^{I_2}(k)<1,
$$
since $\operatorname{I_1}\cap \operatorname{I_2}\neq \emptyset$ implies $b_1<a_{n+1}$.
\end{proof}

\begin{lemma}\label{disjoint}
Let $I_1=\{[a_1,a_2[,\cdots\ ,[a_n,a_{n+1}[\}$, $I_2=\{[b_1,b_2[, \cdots, [b_m,b_{m+1}[\}$ be two right maximal PSI's. Then their supports are either disjoint or nested.
\end{lemma}
\begin{proof}
We proceed by contradiction. Let  $I_1, I_2$ be as above. We can assume that $a_1<b_1$, $\operatorname{ supp}I_1\cap \operatorname{supp}I_2\neq \emptyset$, and $I_2$ is not contained in $I_1$. Thus, $b_1<a_{n+1}$ and $a_{n+1}<b_{m+1}$.

By lemma \ref{union}, $I_1\cup I_2$ should be a piecewise shadowing interval. However, this implies $G^{I_1}(a_{n+1})=G^{I_1}(b_1)G^{I_2}(a_{n+1})<1$, contradicting $I_1$ being maximal.
\end{proof}

Recall that we defined $K=K(f):=\max_{c\in C(f)}\{ d_c \}$.

\begin{lemma}\label{length}
If $I=\{[a_1,a_2[,\cdots ,[a_n,a_{n+1}[\}$ is a right maximal PSI, then
$$
a_{n+1}-a_1\geq \sum_{a\in A\cap [a_1,a_{n+1}[} \left( l(a)-\frac{K+1}{\lambda}\cdot \log2 \right).
$$
\end{lemma}
\begin{proof}
By definition of right maximality,
$$
\begin{aligned}
1\leq G(a_{n+1})&=\prod_{a_1\leq k<a_{n+1}}F( k)\\
&=\prod_{a\in A\cap [a_1,a_{n+1}[}^n (|f'x_{a}|2^{K+1}) \cdot \prod_{k\notin A} \Lambda\\
&\leq \left( \prod_{a\in A\cap [a_1,a_{n+1}[}^n (|f'x_{a}|2^{K+1})  \right) \cdot \Lambda^{a_{n+1}-a_1}
\end{aligned}
$$
Taking logarithm, we get
$$
0\leq \sum_{a\in A\cap [a_1,a_{n+1}[}^n (-\lambda\cdot l(a)+(K+1)\log 2)+(a_{n+1}-a_1)\cdot \lambda.
$$
So,
$$
a_{n+1}-a_1\geq  \sum_{a\in A\cap [a_1,a_{n+1}[} \left( l(a)-\frac{K+1}{\lambda}\cdot \log2 \right).
$$
\end{proof}

\subsection{Inductive Construction of $\mathcal{S}(x)$ }

In this subsection, we produce the sequence $\mathcal{S}(x)$ of intervals announced in theorem \ref{main2} by an inductive scheme. We can assume that $C(f)\cap \operatorname{ supp}\nu\neq \emptyset$, otherwise theorem \ref{main2} is trivial.

We fix $\varepsilon>0$, $L\in \mathbb{N}_+$ and let $\nu\in \mathbb{P}_{\rm{erg}}(f)$, with $h(\nu)>0$. Note that for all $\delta>0$, $\nu(\{ x:|f'x|<\delta \})>0$. Otherwise, $\alpha=0$ and theorem \ref{Main} is trivial.

\begin{definition}
Let $x\in X$, and $I=[a,b[$ be a shadowing interval for $x\in X$ with $b\in A$. We say that {\emph{I is non-switching}}, if 
$$
d(x_b,c_b)\geq d(x_b,f^{b-a}c_a).
$$
Otherwise we say that it is {\emph{switching}}.
\end{definition}

The terminology is explained by the following proposition. 

\begin{proposition}\label{1step}
Let $I=[a,b[$ be a shadowing interval with $b\in A$. 

If $I$ is non-switching, then
$$
d(f x_b,f^{b-a+1} c_a)\leq 2^{K+1}|f'x_b|\cdot d(x_b,f^{b-a}c_a);
$$
if $I$ is switching, then
$$
d(f x_b, f c_b)\leq 2^{K+1}|f'x_b|\cdot d(x_b,f^{b-a} c_a).
$$
\end{proposition}
\begin{proof}
First note that, as $f$ has only non-flat critical points, and $\varepsilon$ is small, $|f'|$ is monotone from $x_b$ to $c_b$ and from $f^{b-a}c_a$ to $c_b$. Thus, by the mean value theorem,
$$
d(x_{b+1},f^{b-a+1}c_a)\leq \max\{|f'x_b| ,|f' f^{b-a} c_a| \}\cdot d(x_b,f^{b-a}c_a).
$$

We consider 3 cases.

{\bf Case A.1:} When I is non-switching and $|f'x_b|\geq |f'f^{b-a}c_a|$.

By the discussion above, 
$$
d(x_{b+1},f^{b-a+1}c_a)\leq |f'x_b| \cdot d(x_b,f^{b-a}c_a)\leq 2^{K+1}|f'x_b| \cdot d(x_b,f^{b-a}c_a).
$$

{\bf Case A.2:} When I is non-switching and $|f'x_b|\leq |f'f^{b-a}c_a|$. 

Then
$$
d(x_{b+1},f^{b-a+1}c_a)\leq |f' f^{b-a} c_a|\cdot d(x_b,f^{b-a}c_a).
$$

Then we compare $|f'f^{b-a}c_a|$ with $|f'x_b|$. Since $I$ is non-switching, $d(f^{b-a}c_a,c_b)<d(f^{b-a}c_a,x_b)+d(x_b,c_b)\leq 2d(x_b,c_b)$. So
$$
\begin{aligned}
\frac{|f' f^{b-a}c_a|}{|f'x_b|}&=\frac{|f^{(d_b)}(c_b)(f^{b-a}c_a-c_b)^{d_b}\cdot g_{c_b}(f^{b-a}c_a)|}{|f^{(d_b)}(c_b)(x_b-c_b)^{d_b}\cdot g_{c_b}(x_b)|}\\
& \leq \frac{|(2(x_b-c_b))^{d_b}|}{|(x_b-c_b)^{d_b}|}\cdot \frac{| g_{c_b}(f^{b-a}c_a)|}{|g_{c_b}(x_b)|}\\
&\leq 2^{d_b}\cdot 2\leq 2^{K+1},
\end{aligned}
$$
where the last inequality follows from $\eqref{eqn:2ast}$, $\eqref{eqn:3ast}$.

{\bf Case B:} When $I$ is switching. 

Since $|f'|$ is decreasing from $x_b$ to $c_b$, we have
$$d(x_{b+1}, f c_b)\leq |f'x_b|d(x_b,c_b).$$
We conclude
$$
\begin{aligned}
d(x_{b+1}, f c_b)&<|f'x_b|d(x_b,c_b)\\
&<|f'x_b|d(x_b,f^{b-a}c_b)< 2^{K+1}|f'x_b|d(x_b,f^{b-a}c_b).
\end{aligned}
$$
\end{proof}

\begin{definition}
Given $x\in X$ and a pair $(I,a^{\ast})$, where $I:=\{[a_1,a_2[,\cdots [a_s,a_{s+1}[ \}$ is a PSI and $a^{\ast}\in A\cap[a_s,a_{s+1}[$. Define
$$
G^I_{a^{\ast}}: [a_1,\infty[ \rightarrow \mathbb{R}_{+}
$$
by
$$
G^I_{a^{\ast}}(k):=\prod_{a_1\leq a\leq k, a\in A_{a^{\ast}}}(|f'x_{a}|2^{K+1})\prod_{a_1\leq l \leq k, l\notin A_{a^{\ast}}} \Lambda,
$$
where $A_{a^{\ast}}:=A\cap [a_1,a^{\ast}[$.

This pair $(I,a^{\ast})$ is {\emph{an intermediate piecewise shadowing interval, abbreviated as an IPSI}}, if we have
\begin{itemize}
\item[1.] for all $k\in \operatorname{supp} I:=[a_1,a_{s+1}[$,
$$
G^I_{a^{\ast}}(k)<1;
$$
\item[2.] if $k\in [a_j,a_{j+1}[$, for some $1\leq j\leq s$, then
$$
d(x_k,f^{k-a_j}c_{a_j})<G^I_{a^{\ast}}(k)\varepsilon.
$$
\end{itemize}
\end{definition}

The terminology is explained in the following proposition.

\begin{proposition}\label{inductive}
Given $\varepsilon>0$, there exists $\delta_1>0$ such that for all $0<\delta<\delta_1$, every $x\in X$ such that $|(f^n)'x|\rightarrow \infty$, for each $a_1\in A$, there exists $I$, a right maximal PSI starting at $a_1$.
\end{proposition}
\begin{proof}

Fix $\varepsilon>0$ small, then lemma \ref{delta} gives $\delta_1>0$. Then, choose a $0<\delta<\delta_1$. Let $x\in X$ and $a_1\in A$. 

{\bf Claim:} There is a finite sequence of pairs $(I_s,a^{\ast}_s)$, $s=1,\cdots S$, with $I_s$ a finite collection of integer intervals and $a^{\ast}_s\in \mathbb{N}$ such that 
\begin{itemize}
\item[0.] $(I_s,a^{\ast}_s)$ is a IPSI, $s=1,\cdots, S-1$;
\item[1.] $\min I_s=a_1$;
\item[2.] for $1<s\leq S$, $\operatorname{supp} I_{s-1}\subseteq  {\operatorname{supp}} I_{s}$, $a^{\ast}_{s-1}<a^{\ast}_s$;
\item[3.] $S>1$ is the minimal $s\geq 1$, such that $a^{\ast}_s=\max I_s+1$;
\item[4.] $I_S$ is a PSI and it is right maximal.
\end{itemize}

\begin{proof}
We build by induction a sequence of IPSI's satisfying conditions $(0)$, $(1)$ and $(2)$ until condition $(3)$ holds. Then, we check condition $(4)$.

For $s=1$, we define $(I_1,a^{\ast}_1)$ by setting $I_1:=\{ [a_1,b_1[ \}$, with $b_1:=a_1+l(a_1)-\frac{K+1}{\lambda}\cdot \log 2,$ and $a^{\ast}_1=a_1$.

Note that $[a_1,b_1[$ is a shadowing interval by lemma \ref{delta}, so $(I_1,a^{\ast}_1)$ is an IPSI.

Now assume by induction that for some $s\geq 1$, there are some IPSI's $(I_1,a^{\ast}_1),\cdots, (I_s,a^{\ast}_s)$ satisfying conditions $(1),(2)$ above, and $a^{\ast}_t\leq \max I_t$ for $t\leq s$. The induction ends when $s+1=\max I_{s+1}+1$.

Now we build a new IPSI according to the following two cases.

\bigskip

{\bf Case A:} $A\cap ]a^{\ast}_s, \max I_s]\neq \emptyset$. 

Write $I_s:=\{ [a_1,a_2[,\cdots, [a_r,b[ \}$. Notice that since $I_s$ is an IPSI, $[a_r,b[$ is a shadowing interval with $a^{\ast}_s\in [a_r,b[$.

Fix $a^{\ast}_{s+1}:=\min\{ a\in A\cap ]a^{\ast}_s,\max I_s] \}\in [a_r,b[$. Note that $[a_r,a^{\ast}_{s+1}[$ is either switching ot non-switching. 

{\bf Case A.1:} $[a_r,a^{\ast}_{s+1}[$, is non-switching. Define
$$
\begin{aligned}
&b':=\min\{ i>b, i\in \mathbb{N}: G^{I_s}_{a^{\ast}_{s+1}}(i)\geq 1 \};\\
&I_{s+1}:=\{ [a_1,a_2[,\cdots,[a_r,b'[ \}.
\end{aligned}
$$

Notice that these two are well defined, since $G^{I_s}_{a^{\ast}_{s+1}}$ is a map defined on $[a_1,\infty[$. We now verify that $(I_{s+1},a^{\ast}_{s+1})$ is indeed an IPSI.

%$a_{s+1}\in [a_{l_s},b_s[$. Notice that since $I_s$ is an intermediate piecewise shadowing interval, $[a_{l_s},b_s[$ is a shadowing interval. So either $[a_{l_s},a_{s+1}[$ is green or it is not.

%{\bf Case B.1:} $a_{s+1}\in [a_{l_s},b_s[$ and $[a_{l_s},a_{s+1}[$ is green. Define
%$$
%\begin{aligned}
%a^{\ast}&:=a_{s+1};\\
%b_{s+1}&:=\min\{ b>a_{s+1}: G(a_1,a^{\ast},b)\geq 1 \};\\
%I_{s+1}&:=[a_1,a_{l_1}[\cup \cdots \cup [a_{l_{s-1}},a_{l_s}[\cup [a_{l_s},b_{j+1}[.
%\end{aligned}
%$$

%We verify that this $(I_{s+1},a_{\ast})$ is an intermediate piecewise shadowing interval.

First notice that by our choice of $b'$, automatically, $G^{I_{s+1}}_{a^{\ast}_{s+1}}(k)< 1$ for all $k\in \operatorname{ supp}I_{s}$.

Note that by our choice $a^{\ast}_{s+1}>a^{\ast}_s$, so $A\cap ]a^{\ast}_{s},a^{\ast}_{s+1}[=\emptyset$. Now since $G^J_{-}(k)$ depends on $A\cap[a_1,-[$, we see that for $a_1\leq k\leq a^{\ast}_{s+1}$,
$$
\begin{aligned}
G^{I_{s+1}}_{a^{\ast}_{s+1}}(k)&=G^{I_{s+1}}_{a^{\ast}_s}(k)\\
&=G^{I_s}_{a^{\ast}_s}(k);
\end{aligned}
$$
and for $a^{\ast}_{s+1}<k<b'$,
$$
\begin{aligned}
G^{I_{s+1}}_{a^{\ast}_{s+1}}(k)&=\prod_{a_1\leq a_j<k, a_j\in A_{a^{\ast}_{s+1}}} \left(|f'x_{a_j}|2^{K+1} \right)\prod_{a\leq j<k, j\notin A_{a^{\ast}_{s+1}}} \Lambda\\
&=G^{I_s}_{a^{\ast}_s}(a^{\ast}_{s+1})(|f'x_{a^{\ast}_{s+1}}|2^{K+1})\Lambda^{k-a^{\ast}_{s+1}-1},
\end{aligned}
$$
for all $1\leq k\leq b'$. So for $a_1\leq a_i \leq k \leq a_{i+1} \leq a^{\ast}_{s+1}$, by inductive hypothesis, $d(x_k,f^{k-a_i}c_{a_i})<G^{I_{s+1}}_{a^{\ast}_{s+1}}(k)\varepsilon$. For $a_r< a^{\ast}_{s+1}<k<b'$,
\begin{equation}
 \tag{2.1}
\begin{aligned}
d(x_k,f^{k-a_i}c_{a_r})&\leq \Lambda^{k-a^{\ast}_{s+1}-1}d(x_{a^{\ast}_{s+1}+1},f^{a^{\ast}_{s+1}-a_r+1}c_{a_r})\\
&\leq  \Lambda^{k-a^{\ast}_{s+1}-1}(|f'x^{\ast}_{a_{s+1}}|2^{K+1})d(x_{a^{\ast}_{s+1}},f^{a^{\ast}_{s+1}-a_r}c_{a_r})\\
&\leq  G^{I_s}_{a^{\ast}_s}(a_{s+1})(|f'x_{a^{\ast}_{s+1}}|2^{K+1})\Lambda^{k-a^{\ast}_{s+1}-1}\varepsilon\\
&=G^{I_s}_{a^{\ast}_{s+1}}(k)\varepsilon,
\end{aligned}
 \label{eqn:ast1}
\end{equation}
where the second inequality follows from  $a^{\ast}_{s+1}\in[a_r,b'[$ being non-switching and proposition \ref{1step}; the third inequality follows from $a^{\ast}_{s+1}\in [a_r,b[$ and inductive assumption; the fourth equality follows from our definition of $G^{I_{s+1}}_{a^{\ast}_{s+1}}$. From \eqref{eqn:ast1}, and the fact that $G^{I_{s+1}}_{a^{\ast}_{s+1}}(k)<1$, it is easy to check that $[a_r, b'[$ is shadowing.

{\bf Case A.2:} $[a_r,b[$ is switching. Define
$$
\begin{aligned}
&b':=\min\{ i>b, i\in \mathbb{N}: G^{I_s}_{a^{\ast}_{s+1}}(i)\geq 1 \};\\
&a_{r'}:=a^{\ast}_{s+1}\\
&I_{s+1}:=\{ [a_1,a_2[,\cdots,[a_r,a^{\ast}_{s+1}[,[a^{\ast}_{s+1},b'[ \};\\
&a_{r+1}:=a^{\ast}_{s+1}.
\end{aligned}
$$

We verify that $I_{s+1}$ is an IPSI.

For $a_1\leq k\leq a^{\ast}_{s+1}$, since by our choice $a^{\ast}_{s+1}>a^{\ast}_s$,
$$G^{I_{s+1}}_{a^{\ast}_{s+1}}(k)=G^{I_s}_{a^{\ast}_s}(k);$$
and for $a^{\ast}_{s+1}<k<b'$,
$$
G^{I_{s+1}}_{a^{\ast}_{s+1}}(k)=G^{I_s}_{a^{\ast}_s}(a^{\ast}_{s+1})(|f'x_{a^{\ast}_{s+1}}|2^{K+1})\Lambda^{k-a^{\ast}_{s+1}-1}.
$$
So for $a_1\leq k\leq a^{\ast}_{s+1}$, naturally $d(x_k,f^{k-a_i}c_{a_i})<G^{I_{s+1}}_{a^{\ast}_{s+1}}(k)\varepsilon$. For $a^{\ast}_{s+1}<k<b'$, as $k\in [a_{r'},b'[$, and $a_{r'}=a^{\ast}_{s+1}$ by our construction,
$$
\begin{aligned}
d(x_k,f^{k-a_{s+1}}c_{a_{s+1}})&\leq \Lambda^{k-a_{s+1}-1}d(x_{a_{s+1}+1},f c_{a_{s+1}})\\
&\leq  \Lambda^{k-a_{s+1}-1}(|f'x_{a_{s+1}}|2^{K+1}))d(x_{a_{s+1}},f^{a_{s+1}-a_{l_s}}c_{a_{l_s}})\\
&\leq  G_{a_s}(a_{s+1})(|f'x_{a_{s+1}}|2^{K+1})\Lambda^{k-a_{s+1}-1}\varepsilon\\
&=G(a_1,a_{s+1},k)\varepsilon\leq \varepsilon,
\end{aligned}
$$
where the second inequality follows from lemma \ref{1step}, applied to $a^{\ast}_{s+1}\in[a_r,b[$, the switching interval. 

The integer intervals $[a_1,a_2[,\cdots [a_{r-1},a_r[$ are all shadowing by inductive assumption. The interval $[a_r,a^{\ast}_{s+1}[$ is also shadowing, since $[a_r,a^{\ast}_{s+1}[\subseteq [a_r,b[$, and the latter is shadowing by inductive assumption. $[a^{\ast}_{s+1},b'[$ is shadowing by the arguments above as well as the definition of $b'$.

\bigskip

{\bf Case B:} $A\cap ]a^\ast_s, \max I_s]=\emptyset$.

In this case we define 
$$
\begin{aligned}
&I_{s+1}:=I_s:=\{ [a_1,a_2[,\cdots [a_r,b[ \};\\
&a^{\ast}_{s+1}:=\max I_s +1;\\
&S:=s+1.
\end{aligned}
$$

By our assumption of this case, $a_{s+1}>b$, for $I_s$,
$$
A_{a^{\ast}}=A\cap [a_1,a_s]=A\cap[a_1,b[.
$$
In this case, for all $a_1<k<b$,
$$
G^{I_{S}}(k)=G^{I_{S}}_{a^{\ast}_s}(k)< 1,
$$
So $I_S$ is a PSI. By the definition of $b$. 

\bigskip

Note that in {\bf Case B}, the constructed sequence $(I_s,a^{\ast}_s)$ satisfies the first three conditions $(0)$, $(1)$, $(2)$ by construction. It remains to check that $(I_s,a^{\ast}_s)$ also satisfies conditions $(3)$ and $(4)$.

Condition $(3)$ follows by induction hypothesis.  

Then we check condition $(4)$. By construction $I$ is a PSI, so it remains to check that it is maximal. Note that by our construction in {\bf Case B}, $I_S=I_{S-1}$. By our definition of $b'$ for $I_{S-1}$, we know that $I_S$  is right maximal.

Now it remains to show that for all $x\in X$ with $|(f^n)'x\rightarrow \infty|$, the inductions ends up with {\bf Case B} within a finite amount of steps.

We prove by contradiction. Note that for each $x\in X$, the collection of visits near the critical set $A$ is a monotone non-decreasing sequence in integers. So by induction, when $s$ goes to infinity, $a^{\ast}_s$ is also monotonely non-decreasing. Thus, $a^{\ast}_{s}\geq s$. So for all $k\geq a_1$, $k\leq a^{\ast}_k<b$, and $k\in [a_1,b[$. 

On the one had, note that by definition of IPSI, for this $k$, 
$$
1> G^{I_k}_{a^{\ast}_k}(k)\geq |(f^{k-a_1})'x_{a_1}|.
$$
On the other hand, we have assumed $|(f^{k-a_1})' x_{a_1}|\rightarrow \infty$. For $k$ large enough, this yields a contradiction.

\bigskip

This finishes the proof of the claim.
\end{proof}

Just take $I=I_S$ and we proved the proposition.

%Claim that given the sequence $\{ I_1,\cdots I_{s+1} \}$ of IPSI's as in the claim, the last IPSI $I_{s+1}$ is actually a right maximal PSI. Since $a^{\ast}_{s+1}:=\max I_{s+1}+1>I_{s+1}$, we have that
%$$
%\begin{aligned}
%G^{I_{s+1}}_{a^{\ast}_{s+1}}(k)&=\prod_{a_1\leq a_j<k, a_j\in A_{a^{\ast}_{s+1}}} \left(|f'x_{a_j}|2^{K+1}\right)\prod_{a\leq j<k, j\notin A_{a^{\ast}_{s+1}}}\\
%&=\prod_{a_1<a_j<k,a_j\in I_{s+1}}\left(|f'x_{a_j}|2^{K+1}\right)\prod_{a\leq j<k, j\notin I_{s+1}}\\
%&=\prod_{a_1\leq j<k}F(\delta,j)
%\end{aligned}
%$$
%for $a_1\leq k<b_{s+1}$, and the later is less than $1$ by definition. The fact that $I_{s+1}$ consists of a sequence of consecutive shadowing interval is clear by construction. Moreover, as we have proved, the induction ends in {\bf Case A}, so $I_s$ can only be of {\bf Case B}. By construction of $b_s=b_{s+1}$, $I_{s+1}$ is right maximal.

%In the special case when $s+1=1$, the proof of lemma \ref{delta} guarantees that $I_{s+1}$ is also a right maximal PSI.

\end{proof}

\subsection{Proof of Theorem 2.1}

We start with one basic lemma from abstract ergodic theory. 

\begin{lemma}\label{asym}
Let $(X,\mathcal{B},\nu)$ be a probability space, $f:X\rightarrow X$ be a $\nu$-ergodic map. Let $\Psi:X\rightarrow [-\infty,\infty)$ be a $L^1(\nu)$-function, such that
$$\int \Psi d\nu \neq 0.$$

If for all $x\in X$, there are two sequences $(a_n)_{n\in \mathbb{N}}$, $(b_n)_{n\in \mathbb{N}}$ such that
\begin{itemize}
\item[1)] $a_n\leq n\leq b_n$, $a_n \uparrow \infty$;
\item[2)] for all $x\in X$,
$$
\begin{aligned}
&\inf_{n\geq 1} \sum_{i=a_n}^{b_n-1}\Psi(x_i)\geq0;\\
&\eta(x):=\sup_{n\geq 1}\sum_{i=a_n}^{b_n-1}\Psi(x_i)<\infty.
\end{aligned}
$$
\end{itemize}
Then for $\nu$-a.e. $x\in X$, $b_n-n=o(n)$, $n-a_n=o(n)$.
\end{lemma}
\begin{proof}
Since $b_n\geq n$, $b_n \rightarrow \infty$ for all $x\in X$.

Since $\nu\in \mathbb{P}_{\rm erg}(f)$, by Birkhoff's ergodic theorem, for $\nu$-a.e. $x\in X$, 
$$\sum_{i=0}^{b_n-1}\Psi (x_i)=b_n\cdot \int \Psi d\nu+o(b_n),$$
$$\sum_{i=0}^{a_n-1}\Psi (x_i)=a_n\cdot \int \Psi d\nu+o(a_n).$$
Subtracting the second equation from the first, we have
\begin{equation}
  \tag{$2.2$}
  \sum_{i=a_n}^{b_n-1}\Psi (x_i)=(b_n-a_n)\cdot \int \Psi d\nu +o(b_n)+o(a_n).
  \label{eqn:1}
\end{equation}

By the first assumption, $a_n\leq b_n$. Hence an $o(a_n)$ is also an $o(b_n)$. Moreover, by the second assumption, $ 0\leq \sum_{i=a_n}^{b_n-1}\Psi x_i\leq \eta$, $ \sum_{i=a_n}^{b_n-1}\Psi x_i=o(b_n)$. Replacing corresponding terms in equation \eqref{eqn:1}, we have
$$
0=(b_n-a_n)\cdot \int \Psi d\nu +o(b_n).
$$

Now that $\int \Psi d\nu\neq 0$ as assumed, we have $b_n-a_n=o(b_n)$.

Since $b_n-a_n=|b_n-n|+|a_n-n|$,
$$b_n-n=o(n)=a_n-n.$$
\end{proof}

Now we prove theorem \ref{main2} using proposition \ref{inductive} and lemma \ref{asym}.

\begin{proof}

Fix $\varepsilon>0$, $L\in \mathbb{N}^+$. 

Fix $\delta_1$ as in proposition \ref{inductive}. Then, take $0<\delta<\delta_1$, such that
$$
-\log \delta> L\cdot (L+\frac{K+1}{\lambda}).
$$ 

Let $\nu\in \mathbb{P}_{\rm{erg}}(f)$ with $\lambda(\nu)>0$.

\bigskip

{\bf Step 1.} The construction of $\mathcal{S}(x)$.

%Since $\lambda(\nu)>0$, $\operatorname{supp }\nu\cap C(f)=\emptyset$. So we may apply proposition \ref{inductive} consecutively to get that 
%, there exists a measurable subset $X_1\subseteq X$, such that $\nu(X_1)=1$ and for all $x\in X_1$, for $a_1(x)=\min A(x)$, there exists a right maximal PSI $J_1(x)$, starting at $a_1(x)$. 
%Then, apply proposition \ref{inductive} consecutively, there exists $X_{l+1}\subseteq X_{l}$ such that $\nu(X_{l+1})=1$, and for all $x\in X_{l+1}$, for $a_{l+1}(x)=\min \{a\in A(x): a> \max J_{l}(x)  \}$, there exists a right maximal piecewise shadowing interval $J_{l+1}(x)$, starting at $a_{l+1}(x)$. Define $X':=\bigcap_{l\in \mathbb{N}^+} X_l$. Then, $\nu(X')=1$, and for all 
% there exists $X'\subseteq X$, $\nu(X')=1$ such that for all $x\in X'$, there exists a collection of finite right maximal PSI's $\mathcal{E}(x):=\{ J_l=J_l(x) \}_{l\in \mathbb{N^+}}$.
 
 This implies $\nu(X_0)=1$ where $X_0:=\{ x\in X: |(f^n)'x|\rightarrow \infty \}$. By proposition \ref{inductive}, there exists $X_1\subseteq X_0$, $\nu(X_1)=1$, such that for all $x\in X_1$, there exists a right maximal PSI $J_1(x)$ , with $ \min J_1(x)=\min A(x)$. Consecutively, given $l\in \mathbb{N}^+$, apply proposition \ref{inductive} again to $X_l$ with $a_{l+1}:=\min\{a\in A(x): a>\max J_l(x) \}$, we get $X_{l+1}$ with $\nu(X_{l+1})=1$ and $J_{l+1}(x)$, PSI with $\min J_{l+1}(x)=a_{l+1}$. Take $X':=\bigcap_{l\geq 0}X_l$. Then $\nu(X')=1$, and for all $x\in X'$, there exists a collection of finite right maximal PSI's $\mathcal{E}(x):=\{ J_l=J_l(x) \}_{l\in \mathbb{N^+}}$.

Fix $x\in X'$. We choose the initial position $a_{l+1}$ to be $ \min \{ a\in A: a>\max J_l \}$. So by lemma \ref{disjoint}, each $J, J'\in \mathcal{E}(x)$ are disjoint. Thus, we may assume $\max J_l<\min J_{l+1}$ for all $l\in \mathbb{N}$. In particular, given $x\in X'$, and $J, J'\in \mathcal{E}(x)$, $J$ and $J'$ cannot be nested. Moreover, for all $J\in \mathcal{E}(x)$, by lemma \ref{length},
$$
\begin{aligned}
\left| \operatorname{ supp} J \right| &\geq \sum_{a_i\in A\cap J}(l(a_i)-\frac{K+1}{\lambda} \log 2)\\
&\geq \sum_{a_i\in A\cap J}(l(a_i)-\frac{K+1}{\lambda}).
\end{aligned}
$$

Remember that each $J\in \mathcal{E}(x)$ is a PSI, which is a collection of shadowing intervals. For each $x\in X'$, define
$$
\mathcal{S}(x):=\{ I=[a,b[\subseteq \mathbb{N}: \exists J(x)\in \mathcal{E}(x) \  s.t. \  I\in J(x), |I|\geq L  \},
$$
that is, $\mathcal{S}(x)$ consists of shadowing intervals in some $J\in \mathcal{E}(x)$ with length greater than $L$.

\bigskip

{\bf Step 2} The 'tail' interval

For the sake of simplicity, we first introduce some notations. Abbreviate $\overline{J}:=\operatorname{supp} J$ for $J\in \mathcal{E}(x)$. For all $n\in \mathbb{N}^+$, let $l_n$ be the maximal integer such that $\min \overline{J}_{l_n}<n$.

Claim that if we define

\begin{equation}
a_n=
\begin{cases*}
\min \overline{J}_{l_n} & if $n\in \overline{J}_{l_n}$ \\
     n        & otherwise 
\end{cases*}
\end{equation}

\begin{equation}
b_n=
\begin{cases*}
\max \overline{J}_{l_n} & if $n\in \overline{J}_{l_n}$\\
n & otherwise
\end{cases*}
\end{equation}

then $b_n-a_n=o(n)$. To prove the claim, we use lemma \ref{asym}.

Define

\begin{equation}
\Psi(x):=
\begin{cases*}
-\lambda(f),  &  if $|f'x|\geq \delta $ \\
-(\log|f'x|+(K+1)\log 2) ,  & if $|f'x|< \delta $.
\end{cases*}
\end{equation}

We check that $a_n, b_n$ and $\Psi$ satisfies all the hypothesis of lemma \ref{asym}.  Firstly, we check that $\int \Psi d\nu \neq 0$. By definition,
$$
\begin{aligned}
\int \Psi d\nu&= \int_{\{ |f'|<\delta \}} -(\log |f'x|+(K+1)\log2)d\nu +\int_{\{|f'|\geq \delta\}}-\lambda(f)d\nu\\
&\leq -\int \log |f'x| d\nu - \int_{\{ |f'|<\delta \}} (K+1)\log2 d\nu<0.
\end{aligned}
$$
The first inequality follows from the fact that $\lambda(f)\geq \log|f'x|$ for all $x\in X$. The last inequality follows from our assumption that $\lambda(\nu)=\int \log|f'|d\nu >0$ and our choice of $\delta$ being small.

When $n\notin \overline{J}_{l_n}$, then $a_n=b_n=n$, and $\sum_{i=a_n}^{b_n-1}\Psi(x_i)=0$. Assume $n\in \overline{J}_{l_n}$, Remember that by definition, $J_{l_n}$ is a right maximal PSI. So $G^{J_{l_n}}(\max \overline{J}_{l_n}+1)\geq 1$, which implies
$$
\begin{aligned}
0&\leq \log G^{J_{l_n}}(\max \overline{J}_{l_n}+1)\\
&=\sum_{j\in \overline{J}_{l_n}\cap A^c}\lambda(f)+\sum_{a\in \overline{J}_{l_n}\cap A}(\log |f'|+(K+1)\log2)\\
&=-\sum_{i=a_n}^{b_n} \Psi(x_i)\\
&=-\sum_{i=a_n}^{b_n-1}\Psi(x_i)+\lambda(f),
\end{aligned}
$$
that is, $\sum_{j=a_n}^{b_n-1}\Psi(x_j)\leq \lambda(f)$. The last equality follows from lemma \ref{last}: $b_n=\max \overline{J}_{l_n}+1\notin A$, so $\log \Psi(x_{b_n})=\lambda(f)$.

Meanwhile, since $J_{l_n}$ is a PSI, $G^{J_{l_n}}(\max \overline{J}_{l_n})\leq 1$. So $\sum_{j=a_n}^{b_n-1}\Psi(x_j)\geq 0$. So $a_n$, $b_n$ and $\Psi$ satisfies all the requirements of lemma \ref{asym}.

\bigskip

{\bf Step 3.} Verifying the properties in theorem \ref{main2}.

0. Since elements in $\mathcal{S}(x)$ are constructed in proposition \ref{inductive} by visits to critical points, the map $x\mapsto \mathcal{S}(x)$ is naturally measurable.

1. Note that since elements in $\mathcal{E}(x)$ are finite, pairwise disjoint, so will elements be in $\mathcal{S}(x)$. Moreover, by construction, elements in $\mathcal{S}(x)$ are naturally shadowing intervals.

2. Let
$$
\psi(x):=1_{\{ |f'|<\delta \}}(\frac{-\log|f'|+K+1}{\lambda(f)}).
$$

Note that $\psi(x)\geq 0$ and $\psi(x_j)>0$ if and only if $j\in \bigcup \mathcal{E}(x)$. In particular,
$$
\sum_{i=1}^{l_n-1} |\overline{J}_i|\leq \sum_{i=0}^{n-1} \psi(x_i)\leq \sum_{i=1}^{l_n} |\overline{J}_i|.
$$
Thus, $\sum_{i=1}^{l_n-1} |\overline{J}_i|=\sum_{i=0}^{n-1} \psi(x_i)+o(n)$, by {\bf step 2}.

Now 
$$
\sum_{I\in \mathcal{S}(x), I^{-}<n} |I|\leq \sum_{i=0}^{n-1}(L\cdot 1_{\{ |f'|<\delta \}}).
$$

So by our choice of $\delta$,
$$
\begin{aligned}
|\mathcal{S}(x)\cap [0,n[|&=\sum_{I\in \mathcal{S}(x), I^{-}<n} |I|\geq \sum_{i=1}^{l_n-1} |\overline{J}_{i}|-\sum_{i=0}^{n} (L\cdot 1_{\{ |f'|<\delta \}})\\
&\geq \sum_{i=0}^{n-1}\psi(x_i) -\sum_{i=0}^{n} (L\cdot 1_{\{ |f'|<\delta \}})+o(n).
\end{aligned}
$$
And that concludes the theorem. 
\end{proof}

\section{The Entropy Gap}

\subsection{Conditions for Entropy Gap}
In this section, we summarize our previous construction and deduce an entropy gap. 

Let $X$ be a compact metric space, and $f:X\rightarrow X$ be a continuous map, with $h_{\rm{top}}(f)>0$. Fix $\varepsilon>0$.

A system of coverings is a family $\{ X_I \}$ where $I=[a,b[\subseteq \mathbb{N}$ ranges over integer intervals, and each $X_I\subseteq X$ is a finite set. The entropy of this system $\{ X_I \}$ is defined as
$$
h(\{ X_I \}):=\limsup_{n\rightarrow \infty}\max_{|I|=n}\frac{1}{n}\log|X_I|.
$$

Take a measurable function $\varphi=\varphi_{L}: X\rightarrow \mathbb{N}\cup \{ \infty \}$. We say that a subset $X'\subseteq X$ admits a partial shadowing by the system of covering $\{X_I\}$, with minimum length $L$, controlled by $\varphi$, if for all $x\in X'$, there is a collection of disjoint integer intervals, denoted by $\mathcal{E}(x)=\mathcal{E}_{L}(x)$, such that
\begin{itemize}
\item[0.] each $I\in \mathcal{E}(x)$ is finite, with $|I|\geq L$;
\item[1.] for each $I=[a,b[\in \mathcal{E}(x)$, there exists $z\in X_{I}$, such that $x_a\in B(z,\varepsilon,|I|)$;
\item[2.] the map $x\mapsto \mathcal{E}(x)$ is measurable on $X$: for each collection of subsets $J$, $\mathcal{E}^J_n(x):=\{ x\in X': \{ x\in X': \{I\cap [0,n[: I \in \mathcal{E}(x)\}=J \} \}$ is measurable;
\item[3.] for all $x\in X'$,
$$
%\begin{aligned}
\liminf_{n\rightarrow \infty}\frac{1}{n} \left(\sum_{I\in \mathcal{E}_n(x)}|I|-(1-\frac{1}{L})\sum_{i=0}^{n-1}\varphi (x_i)\right)\geq 0,
%&\limsup_{n\rightarrow \infty}\frac{1}{n}\left( \sum_{i=0}^{n-1}\varphi (x_i)-\sum_{I\in \mathcal{E}_n(x)}|I|\right)\geq 0.
%\end{aligned}
$$
where $\mathcal{E}_n(x):=\{ I\cap [0,n[: I\in \mathcal{E}(x) \}$.
\end{itemize}

We now state and prove an entropy gap given by partial shadowing. Before that, we recall some notations.

To begin with, we state and prove a lemma concerning combinatorics of subintervals inside $[0,n[$. This is the reason we take subintervals of length greater than $L$ in the definition of partial shadowings.

\begin{definition}
Given $L\in \mathbb{N}^{+}$, for all $n\in \mathbb{N}$, denote by
$$
\mathcal{T}_{L}(n):=\{ \{I_j\}: I_j\subseteq [0,n[, I_{j_1}\cap I_{j_2}=\emptyset, |I_j|\geq L \}
$$
Call the elements in $\mathcal{T}_{L}(n)$ the {\emph{combinatorical types in $[0,n[$}}.
\end{definition}

\begin{lemma}\label{type}
Denote $H(t):=t\cdot\log\frac{1}{t}+(1-t)\cdot \log\frac{1}{1-t}$, for $t\in (0,1)$. Then for all $L\geq 4$, 
$$| \mathcal{T}_{L}(n)| \leq \exp(n\cdot H(\frac{2}{L})+o(n)).$$
\end{lemma}
\begin{proof}
Notice that $H(t)$ is a function defined on $(0,1)$, increasing from $0$ to $\frac{1}{2}$ and decreasing from $\frac{1}{2}$ to $1$.

By definition, every finite collection in $\mathcal{T}_L(n)$ consists of disjoint intervals in $[0,n[$ of lengths at least $L$. Thus, there can be at most $\llcorner \frac{n}{L} \lrcorner$ such intervals maximal in the union. These gives
$$ 
\sum_{i=0}^{\llcorner \frac{n}{L} \lrcorner}\binom{n}{2i}
$$
 many possible types. By assumption, $ \frac{1}{L}<\frac{1}{4}$. So each $\binom{n}{2i}\leq \binom{n}{2\llcorner \frac{n}{L} \lrcorner}$, and
$$
\begin{aligned}
\sum_{i=0}^{\llcorner \frac{n}{L} \lrcorner} \binom{n}{i}
&\leq
\frac{1}{2}\cdot 2\llcorner \frac{n}{L} \lrcorner\cdot
\binom{n}{2\llcorner \frac{n}{L} \lrcorner}\\
&\leq \frac{1}{2}\exp(n\cdot H(\frac{2}{L})+o(n))
\end{aligned}
$$
using de Moivre's approximation. 
\end{proof}

From the definition of topological entropy, we see the following lemma.

\begin{lemma}\label{Katok}
For all $\gamma>0$, $\varepsilon>0$, there exists $C_{\gamma,\varepsilon}>0$ such that for all $n\in \mathbb{N}^{+}$,
$$r(f,n,\frac{\varepsilon}{2},X)\leq C_{\gamma,\varepsilon}\cdot \exp(n\cdot (h_{\rm top}(f)+\gamma)).$$
\end{lemma}

\begin{lemma}\label{submul}
If $N=\sum n_i$, where $N$, $n_i\in \mathbb{N}^{+}$, then
$$r(f,N,\varepsilon,X)\leq \prod r(f,n_i,\frac{\varepsilon}{2}, f^{n_1+\cdots+n_{i-1}}X).$$
\end{lemma}
\begin{proof}
Simply notes that given a $(\frac{\varepsilon}{2},n)$-Bowen ball $B(z,n,\frac{\varepsilon}{2})$ of $f^m X$ and a covering $\mathcal{B}$ of $X$ by $(\frac{\varepsilon}{2},m)$-Bowen balls, each intersection $f^{-m} B(z,n,\frac{\varepsilon}{2})\cap B$ with $B\in \mathcal{B}$ is contained in a $(\varepsilon,m+n)$-Bowen ball.
\end{proof}

\begin{theorem}\label{gap}
Let $f:X\rightarrow X$ be a continuous map over a compact metric space.

Given $\varepsilon>0$, $0<\gamma<\frac{1}{2}$, $L\in \mathbb{N}_{+}$, assume there exists a partial shadowing $(X',\varepsilon,L,\varphi,\{ X_I \},\mathcal{E}(x))$. For all $\nu\in \mathbb{P}_{\rm erg}(f)$, if $\nu(X')>0$, then
$$
h(\nu,\varepsilon)\leq h_{\rm{top}}(f)-(1-\frac{1}{L})\cdot \int \varphi d\nu\cdot (h_{\rm{top}}(f)-h(\{X_I\})+K+2\gamma,
$$
where $K=K(\varepsilon,\gamma,L)$, and for all $\varepsilon,\gamma>0$,
$$
\lim_{L\rightarrow \infty}K(\varepsilon,\gamma,L)=0.
$$
\end{theorem}
\begin{proof}

Given $\varepsilon, \gamma$, let $(X',\varepsilon,L,\varphi, \{ X_I \}, \mathcal{E}(x))$ be a partial shadowing on $(X,f)$. Take $C:=C_{\gamma, \varepsilon}$  as in lemma \ref{Katok} such that
$$
r(f,n,\frac{\varepsilon}{2},X)\leq C \cdot \exp(n\cdot (h_{\rm top}(f)+\gamma)).
$$
Since we have defined $h(\{ X_I \}):=\limsup_{n\rightarrow \infty}\max_{|I|=n}\frac{1}{n}\log|X^{\varepsilon}_I|$, for $\gamma>0$ given, there exists $D=D_{\gamma}$ such that for all $n\in \mathbb{N}^+$,
$$
\max_{|I|=n}|X_I|\leq D \cdot \exp( n\cdot (h(\{ X_I\})+\gamma) ).
$$

Take one $\nu\in \mathbb{P}_{\rm erg}(f)$. 

\bigskip

{\bf Step 1.} Finding a suitable Borel subset.

By definition of a partial shadowing, the function $\varphi$ is strictly positive on $X$. So there exists a Borel subset $Z\subseteq X'$ such that $\nu(Z)=\nu(X')$ and for all $x\in Z$, 
$$
\lim_{n\rightarrow \infty}\frac{1}{n}\sum_{i=0}^{n-1}\varphi(x_i)=\int \varphi d\nu.
$$

We then subdivide $Z$ into smaller pieces. Define the $n$-combinatorical types of $\mathcal{E}(x)$ to be 
$$
T(\mathcal{E}(x),n):=\bigcup_{I\in \mathcal{E}(x)} I\cap [0,n[.  
$$
Note that $T(\mathcal{E}(x),n)\in \mathcal{T}_L(n)$ by the 0-th assumption on partial shadowing. For a given type $T\in \mathcal{T}_L(n)$,  Let 
$$
Z^T:=\{ x\in Z: T(\mathcal{E}(x),n)=T \}.
$$

By the second assumption on partial shadowing, $Z^T$ is a measurable subset of $X$. By lemma \ref{type}, 
$$
|\mathcal{T}_L(n)|\leq \exp(n\cdot H(\frac{2}{L})+o(n)).
$$
Thus for all $n\in \mathbb{N}^+$,
$$
r(f,n,\varepsilon,Z)\leq \exp(n\cdot H(\frac{2}{L})+o(n)) \cdot \max_{T\in \mathcal{T}(x,n)} \{ r(f,n,\varepsilon,Z^T) \}.
$$

\bigskip

{\bf Step 2.} Entropy estimate.

By construction, given each $Z^T$, combinatorial type for $x\in Z^T$ is a fixed union of disjoint intervals $\mathcal{E}_n(x)=T$. Denote by $T^c$ the collection of integer intervals which are complements of those in $T$.

Notice that
$$
\begin{aligned}
h(f,\nu,\varepsilon)&=\limsup_{n\rightarrow \infty} \frac{1}{n} \log r(f,n,\varepsilon,Z)\\
&\leq \limsup_{n\rightarrow \infty}\frac{1}{n}\log(\exp(n\cdot H(\frac{2}{L})+o(n))\cdot \max_{T\in \mathcal{T}} \{ r(f,n,\varepsilon_0,Z^T) \}).
\end{aligned}
$$
Meanwhile, by our definition of $Z$, and lemma \ref{submul}, for all $n\in \mathbb{N}^+$,
$$
\begin{aligned}
& \max_{T\in \mathcal{T}} \{ r(f,n,\varepsilon_0,Z^T) \}\\
&\leq \prod_{[a,a'[\in T}r(f,a'-a,\frac{\varepsilon}{2},f^a Z^T)\prod_{[b,b'[\in T^c} r(f,b'-b,\frac{\varepsilon}{2},f^b Z^T).
\end{aligned}
$$

Since each $[a,a'[$ is part of the system of covering, $r(f,a-a',\frac{\varepsilon_0}{2},f^a Z^T_k)\leq \max_{|I|=n}|X_I|$, as $f^a Z^T$ could be covered by a union of $(a'-a,\varepsilon)$-Bowen balls. Thus we may apply lemma \ref{Katok} to see that
$$
\begin{aligned}
&  \max_{T\in \mathcal{T}} \{ r(f,n,\varepsilon_0,Z^T) \}\\
&\leq  \left(  \prod_{[a,a'[\in T} \max_{|I|=n}|X_I|  \right)\cdot \left( \prod_{[b,b'[\in T^c} r(f,b'-b,\frac{\varepsilon}{2},f^b Z^T) \right)
\end{aligned}
$$
where
$$
\begin{aligned}
& \prod_{[a,a'[\in T} \max_{|I|=n}|X_I| \leq \exp( (h(\{ X_I \})+\gamma)\cdot \left(\sum_{I\in T} |I\cap [0,n[| \right)+|T|\log D)\\
& \prod_{[b,b'[\in T^c} r(f,b'-b,\frac{\varepsilon}{2},f^b Z^T) \leq \exp((h_{{\rm top}}(f)+\gamma)\cdot \left( \sum_{I\in T^c}|I\cap [0,n[| \right)+|T^c|\log C),
\end{aligned}
$$
for all $x\in Z^T$.

Note that, by our assumption 3 and assumption 0, for all $x\in Z^T$, $T=\mathcal{E}_n(x)$, and
$$
\begin{aligned}
&\liminf_{n\rightarrow \infty}\frac{1}{n} \left(\sum_{I\in \mathcal{E}_n(x)}|I|-(1-\frac{1}{L})\sum_{i=0}^{n-1}\varphi(x_i)\right)\geq 0,\\
&\limsup_{n\rightarrow \infty} \frac{|\mathcal{E}_n(x)|}{n}\leq \frac{1}{L}.
\end{aligned}
$$

Thus,
$$
h(f,\nu,\varepsilon) \leq H(\frac{2}{L})+ \frac{\log D}{L}+\frac{\log{C}}{L}+h_{\rm{top}}(f)-\left( (1-\frac{1}{L})\int \varphi d\nu \right)(h_{\rm{top}}(f)-h(\{ X_I \}))+2\gamma.
$$

Just let
$$
K(\varepsilon,\gamma,L):=\lim_{n\rightarrow \infty} H(\frac{2}{L})+\frac{\log C}{L}+\frac{\log D}{L}.
$$
With lemma \ref{type}, one verifies easily that $K(\varepsilon,\gamma,L)\rightarrow 0$ when $\varepsilon$ and $\gamma$ are fixed and $L\rightarrow \infty$.

\end{proof}

\subsection{Proof of the Main Theorem}

Now we use theorem \ref{gap} to prove the main theorem, theorem \ref{Main}.

\begin{proof}
First fix $\gamma_0>0$. 

Then, take $\varepsilon_0$ small, so that $h(f,\nu,\varepsilon_0)<h(f,\nu)+\gamma_0$ for all $\nu\in \mathbb{P}_{\rm{erg}}(f)$. This is possible since $C(f)$ is a finite set and $f$ is $C^{\infty}$. Thirdly, note that by theorem \ref{gap}, with $\varepsilon_0, \gamma_0$ fixed,
$$
\lim_{L\rightarrow \infty} K(\varepsilon_0,\gamma_0,L)=0.
$$
So we may take $L_0$ large so that $K(\varepsilon_0,\gamma_0,L_0)<\gamma_0$, and $\frac{1}{L_0}<\gamma_0$.

By assumption, $\alpha:=\lim_{\delta\to0}\limsup_{k\to \infty}\int_{\{ |f'|<\delta \}} -\log|f'|d\mu_k>0$. So we may take $\delta_0>0$ and take a subsequence if necessary, so that for all $k\in \mathbb{N}_{+}$,
$$
\frac{1}{\lambda}\int_{\{ |f'|<\delta_0 \}} -\log|f'|d\mu_k>\frac{\alpha}{\lambda}-\gamma_0.
$$

Finally, shrink $\delta_0$, so that we get the sequence $\mathcal{S}(x)$ as in theorem \ref{main2} for $([0,1],f,\varepsilon_0,L_0)$. Then, define
$$
\varphi_0(x):=-1_{\{ x:|f'x|<\delta_0 \}}\frac{\log |f'(x)|}{\lambda(f)}.
$$
Take one $\mu_k$. We check that the $\mathcal{S}(x)$ satisfies all the conditions for being partially shadowing. Take $X'$ to be the subset such that
\begin{itemize}
\item[1.] $\mu_k(X')>\frac{1}{2}$;
\item[2.] $|\mathcal{S}(x)\cap[0,n[|\geq (1-\frac{1}{L_0})\sum_{i=0}^{n-1}\varphi_0(x_i)-g_{n}(x)$.
\end{itemize}
The existence of $X'$ is guaranteed by theorem \ref{main2}.

For each $I$, integer interval, define $ X_I $ to be $C(f)$. In this case $h(\{X_I\})=0$. We now verify that $(X',\varepsilon_0,L_0,\varphi_0,C(f),\mathcal{S}(x))$ really gives a partial shadowing for $(f,\mu_k)$. 
\begin{itemize}
\item[0.] by construction, for all $I\in \mathcal{S}(x)$, $|I|\geq L_0$;
\item[1.] by construction, $I=[a,b[\in \mathcal{S}(x)$ is $\varepsilon$-shadowing, so $x_a\in B(c,\varepsilon,|I|)$ for some $c\in C(f)$;
\item[2.] this is obvious, as $I$ is characterized by visiting to $\{ |f'x|<\delta_0 \}$;
\item[3.] this is property $2$ of theorem \ref{main2}.
\end{itemize}

Now apply theorem \ref{gap} to $(X',f,\varepsilon_0.L_0,\varphi_0,\mu_k)$. We get
$$
h(\mu_k,\varepsilon_0)\leq (1-(1-\gamma_0)(\frac{\alpha}{\lambda}-\gamma_0)) h_{\rm{top}}(f)+2\gamma_0.
$$
So
$$
\begin{aligned}
h_{\rm{top}}(f)-h(\mu_k)&\geq  h_{\rm{top}}(f)-h(\mu_k,\varepsilon_0)-\gamma_0\\
&\geq  h_{\rm{top}}(f)-(1-(1-\gamma_0)(\frac{\alpha}{\lambda}-\gamma_0))h_{\rm{top}}(f)-3\gamma_0\\
&\geq \frac{\alpha}{\lambda} h_{\rm{top}}(f)-4\gamma_0-\gamma_0(\frac{\alpha}{\lambda}-\gamma_0)h_{\rm{top}}(f).
\end{aligned}
$$

As the choice of $\gamma_0$ is arbitrary, we get the main theorem.
\end{proof}


\begin{thebibliography}{XXW}
\bibitem[B]{tail}  Buzzi Jérôme. "Intrinsic ergodicity of smooth interval maps." \emph{Israel Journal of Mathematics} {\bf 100.} (1997): 125-161.
%\bibitem[BBC]{BBC} Backes Lucas, Brown aaron and Butler Clark. "Continuity of Lyapunov exponents for cocycles with invariant holonomies." \emph{J. Mod. Dyn.} {\bf 12}(2018): 223-260.
\bibitem[BCS]{BCS-Exponents}  Buzzi Jérôme, Sylvain Crovisier, and Omri Sarig. "Continuity properties of Lyapunov exponents for surface diffeomorphisms." \emph{Inventiones Mathematicae} {\bf 230.2} (2022): 767-849.
%\bibitem[KP]{KP} Kadyrov Shirali, Phol Anke. "Amount of failure of upper-semicontinuity of entropy in non-compact rank-one situations, and Hausdorff dimension." \emph{Ergodic Theory and Dynamical Systems} {\bf 37(2)}(2017):539-563.
\bibitem[M1]{M1} Misiurewicz M., Szlenk W. "Entropy of piecewise monotone mappings." \emph{Systèmes dynamiques II - Varsovie, Astérisque} {\bf 50} (1977): 299-310.
\bibitem[V1]{V1} Viana Marcelo. "(Dis)continuity of Lyapunov exponents." \emph{Ergodic Theory and Dynamical Systems} {\bf 40(3)} (2020):577-611.
\bibitem[V2]{V2} Viana Marcelo. "Lectures on Lyapunov exponents." (Cambridge University Press, 2014)
\bibitem[Q]{Q} Qian Ming, Xie Jiansheng, Zhu Shu. "Smooth ergodic theory for endomorphisms." (Springer Berlin Heidelberg, 2009)
\bibitem[S]{S} Sánchez, Adriana. "The Continuity Problem of Lyapunov Exponents." In: Dias, J.L., et al. New Trends in Lyapunov Exponents. NTLE 2022. CIM Series in Mathematical Sciences. (Springer, Cham, 2023)
\end{thebibliography}
\end{document}